\documentclass[12pt]{article}
\usepackage{amsmath, amsthm, amscd, amsfonts, amssymb, graphicx, color}
\usepackage{epstopdf}
\usepackage{subfigure}
\usepackage{setspace}
\usepackage[left=1in, right=1in, bottom=1.4in]{geometry}
\usepackage{bm}
\usepackage{cite}
\newtheorem{theorem}{Theorem}[section]

\newtheorem{lemma}[theorem]{Lemma}

\newtheorem{example}{Example}[section]
\newtheorem{proposition}[theorem]{Proposition}

\title{\bf Several properties of a class of generalized harmonic mappings \vspace{-2em}}
\date{}
 
\begin{document}
\maketitle

\renewcommand{\thefootnote}{\fnsymbol{footnote}}
\noindent{\footnotesize\rm{ \begin{center}\bf Bo-Yong Long \quad Qi-Han  Wang\end{center}\
 \begin{center}  School of Mathematical Sciences, Anhui University, Hefei  230601, China
\end{center}}
\footnotetext{\hspace*{-5mm}
\begin{tabular}{@{}r@{}p{16.0cm}@{}}
&Supported by   NSFC (No.12271001) and  Natural Science Foundation of Anhui Province
(2308085MA03),  China.\\
&E-mail:  boyonglong@163.com  \quad qihan@ahu.edu.cn

\end{tabular}}

\begin{quote}
{\small \noindent {\bf Abstract}:\  We call the solution of a kind of second order homogeneous  partial differential equation as real kernel $\alpha-$harmonic mappings. In this paper, the representation theorem, the Lipschitz continuity, the univalency and the related problems   of the real kernel $\alpha-$harmonic mappings  are explored.
}\\
  {{\small \noindent{\bf Keywords}:  Weighted Laplacian operator; univalency;  Polyharmonic mappings;   Lipschitz continuity; Gauss hypergeometric function }
 }
\end{quote}

\section{Introduction}
\setcounter{equation}{0}

\hspace{2mm}

Let $\mathbb{D}$ be the open unit disk and  $\mathbb{T}$ the unit circle.
For $\alpha\in \mathbb{R}$ and  $z\in \mathbb{D}$, let

$$T_{\alpha}=-\frac{\alpha^{2}}{4}(1-|z|^{2})^{-\alpha-1}+\frac{\alpha}{2}(1-|z|^{2})^{-\alpha-1}\left(z\frac{\partial}{\partial z}
+\bar{z}\frac{\partial}{\partial\bar{z}}\right)+(1-|z|^{2})^{-\alpha}\triangle$$
be the second order elliptic  partial differential operator, where $\triangle$ is the usual complex Laplacian operator
$$\triangle:=4\frac{\partial^{2}}{\partial z \partial \bar{z}}=\frac{\partial^{2}}{\partial x^{2}}+\frac{\partial^{2}}{\partial y^{2}}. $$
The corresponding  partial differential equation is
\begin{align}\label{1.1}
T_{\alpha}(u)=0 \quad \mbox {in}  \,\mathbb{D}.
\end{align}
The associated Dirichlet boundary value problem  is

\begin{equation}\label{1.2} \;  \left\{
\begin{array}{rr}
T_{\alpha}(u)=0 &\quad \mbox {in}  \,\mathbb{D},\quad \\
u=u^{*} &\quad \mbox {on}  \,\mathbb{T}.
\end{array}\right.
\end{equation}
Here, the boundary data $u^{*}\in\mathfrak{D}'(\mathbb{T})$
is a distribution on the boundary  of $\mathbb{D}$, and the boundary condition in (\ref{1.2}) is interpreted in the distributional sense that $u_{r}\rightarrow u^{*}$ in $\mathfrak{D}'(\mathbb{T})$
as $r\rightarrow1^{-}$, where
\begin{equation*}
u_{r}(e^{i\theta})=u(re^{i\theta}), \quad e^{i\theta}\in\mathbb{T},
\end{equation*}
for $r\in[0,1)$. In \cite{MR3233580}, Olofsson proved that, for the parameter $\alpha>-1$, if a function $u\in\mathcal{C}^{2}(\mathbb{D})$ satisfies (\ref{1.1}) with $\lim_{r\rightarrow 1^{-}}u_{r}=u^{*}\in\mathfrak{D}'(\mathbb{T})$
, then it has the form of Poisson type integral
\begin{equation}\label{1.3}
u(z)=\frac{1}{2\pi}\int_{0}^{2\pi}K_{\alpha}(ze^{-i\tau)})u^{*}(e^{i\tau})d\tau, \quad \mbox{for}\,\, z\in\mathbb{D},
\end{equation}
where
 \begin{equation}\label{1.4}
 K_{\alpha}(z)=c_{\alpha}\frac{(1-|z|^{2})^{\alpha+1}}{|1-z|^{\alpha+2}},
 \end{equation}
$c_{\alpha}=\Gamma^{2}(\alpha/2+1)/\Gamma(1+\alpha)$ and $\Gamma(s)=\int_{0}^{\infty}t^{s-1}e^{-t}dt$ for $s>0$ is the standard Gamma function.
If  $\alpha\leq -1$, $u\in\mathcal{C}^{2}(\mathbb{D})$  satisfies (1.1), and the boundary limit $u^{*}=\lim_{r\rightarrow 1^{-}}u_{r}$ exists in $\mathfrak{D}'(\mathbb{T})$, then $u(z)=0$ for all $z\in \mathbb{D}$. So, in the following of this paper, we always  assume that $\alpha>-1$.

For $c\neq 0, \,-1, \,-2,...$, the  Gauss hypergeometric function is defined by the  series
$$F(a,b;c; x)=\sum_{n=0}^{\infty}\frac{(a)_{n}(b)_{n}}{(c)_{n}}\frac{x^{n}}{n!}$$ for $ |x|<1$,  and has a continuation to the complex plane with branch points at $1$ and $\infty$,
where $(a)_{0}=1$ and $(a)_{n}=a(a+1)\cdots(a+n-1)$ for $n=1,2,...$ are the Pochhammer symbols. Obviously, for $n=0, 1,2,...$,
$(a)_{n}=\Gamma(a+n)/\Gamma(a)$.  It is easily to verified that

\begin{equation}\label{1.5}
\frac{d}{dx}F(a,b;c;x)=\frac{ab}{c}F(a+1,b+1;c+1; x).
\end{equation}
Furthermore,  it holds that (cf.\cite{MR1688958} )

\begin{equation}\label{1.6}
\lim_{x\rightarrow 1}F(a,b;c; x)=\frac{\Gamma(c)\Gamma(c-a-b)}{\Gamma(c-a)\Gamma(c-b)}
\end{equation}if $Re(c-a-b)>0$.

The following Lemma 1.1 involves the determination of monotonicity of Gauss hypergeometric functions.

\begin{lemma}\cite{MR3233580}\label{lem1.1}
Let $c>0$, $a\leq c$, $b\leq c$ and $ab\leq 0$ $(ab\geq 0)$. Then the function $F(a,b;c;x)$ is decreasing (increasing) on $x\in (0, 1)$.
\end{lemma}

 The following result of \cite{MR3233580} is the homogeneous expansion of solutions of (1.1).

 \begin{theorem}\cite{MR3233580}\label{thm1.2}
 Let $\alpha\in\mathbb{R}$ and $u\in\mathcal{C}^{2}(\mathbb{D})$. Then $u$ satisfies (1.1) if and only if it has a series expansion
 of the form
 \begin{equation}\label{1.7}
 u(z)=\sum_{k=0}^{\infty}c_{k}F(-\frac{\alpha}{2},k-\frac{\alpha}{2}; k+1; |z|^{2})z^{k}+\sum_{k=1}^{\infty}c_{-k}F(-\frac{\alpha}{2},k-\frac{\alpha}{2}; k+1; |z|^{2})\bar{z}^{k}, \quad z\in \mathbb{D},
 \end{equation}
 for some sequence $\{c_{k}\}_{-\infty}^{\infty}$ of complex number satisfying
 \begin{equation}\label{1.8}
 \lim_{|k|\rightarrow\infty}\sup|c_{k}|^{\frac{1}{|k|}}\leq 1.
 \end{equation}
 In particular, the expansion (\ref{1.7}), subject to (\ref{1.8}), converges in $\mathcal{C}^{\infty}(\mathbb{D})$, and every solution $u$ of (\ref{1.1}) is $\mathcal{C}^{\infty}-$smooth in $\mathbb{D}$.

 \end{theorem}

Let \begin{equation}\label{1.9}
v(z)=\sum_{k=0}^{\infty}c_{k}z^{k}+\sum_{k=1}^{\infty}c_{-k}\bar{z}^{k}, \quad z\in \mathbb{D}.
\end{equation}
 It is obvious that $v(z)$ is a harmonic mapping, i.e., $\triangle v=0$. We observe that $u(z)$ of (\ref{1.7}) and $v(z)$  have same coefficient sequence $\{c_{k}\}_{-\infty}^{\infty}$. Actually, if $\alpha=0$, then $u(z)=v(z)$.

Observe that  the kernel $K_{\alpha}$ in (\ref{1.4}) is real.  We call $u$  of (\ref{1.3}) or (\ref{1.7}) as {\bf real kernel $\alpha-$harmonic mappings}.
Furthermore, suppose  $u(z)$ and $v(z)$ have the expansions of (\ref{1.7}) and (\ref{1.9}), respectively.
We call $v(z)$ as {\bf the corresponding harmonic mapping } of $u(z)$. Conversely, we call $u(z)$  as {\bf the corresponding real kernel $\alpha-$harmonic mapping} of $v(z)$.

If we take $\alpha=2(p-1)$, then a real kernel $\alpha-$harmonic mapping $u$ is polyharmonic (or $p-$harmonic), where $p\in\{1,2,...\}$ (cf.\cite{MR2386451, MR2684461,  MR3415733,MR3311963, MR3549984, MR3632679, MR3638453, MR3071752}).
In particular, if $\alpha=0$, then $u$ is harmonic (cf.\cite{MR2048384, MR3852565, MR3635221, MR4359800}).  Thus, the real kernel $\alpha-$harmonic mapping is a kind of generalization of classical harmonic mapping.  Furthermore, by \cite{MR3055593}, we know that it is related to standard weighted harmonic mappings.  For the related discussion on standard weighted harmonic mappings, see \cite{MR3535757,  MR3718207, MR3914170, MR4270171}.

For the real kernel $\alpha-$harmonic mappings, the Schwarz-Pick type estimates and coefficient estimates are obtained in \cite{MR3365859}; the starlikeness, convexity and Landau type theorem are studied in \cite{MR4365391}; the sharp Heinz type inequality   is established and the
extremal functions of Schwartz type lemma are explored in \cite{Longbo}; the Lipschitz continuity  with respect to the distance ratio metric is proved in \cite{MR3864837}. In \cite{Khalfallah}, using the properties of the real kernel $\alpha-$harmonic mappings, the authors established some Schwarz type lemmas for mappings satisfying a class of
inhomogeneous biharmonic Dirichlet problem.

In  this paper, we  continue to study the properties of the real kernel $\alpha-$harmonic mappings. The main idea of this paper is that by establishing the relationship between harmonic mapping  and the  corresponding  real kernel $\alpha-$harmonic mapping, we  use the harmonic mapping to characterize the corresponding real kernel $\alpha-$harmonic mapping.
 In section 2, for a nonnegative even number $\alpha$, we get an explicit representation theorem which determines the relation between the real kernel $\alpha-$harmonic mapping and the corresponding harmonic mapping. As its application, in section 3, we show that the Lipschitz continuity of a real kernel $\alpha-$harmonic mapping is determined by the corresponding harmonic mapping. In section 4,   for a subclass of the real kernel $\alpha-$harmonic mappings, we discuss its univalency and explore its Rad\'{o}-Kneser-Choquet type theorem.  In  section 5, we explore the influence of parameters $\alpha$ on the image area of the real kernel $\alpha$-harmonic mappings.

\section{Representation theorem}
\setcounter{equation}{0}
\hspace{2mm}

\begin{theorem}\label{thm 2.1}
 Let $v(z)=h(z)+\overline{g(z)}=\sum_{k=0}^{\infty}c_{k}z^{k}+\sum_{k=1}^{\infty}c_{-k}\bar{z}^{k}$ be a harmonic mapping defined on the unit disk $\mathbb{D}$.
If $\frac{\alpha}{2}=p-1$ is a nonnegative integer, then the corresponding real kernel $\alpha-$harmonic mapping of $v(z)$ can be represented by
  \begin{align}\label{2.1}
  u(z)=\sum_{n=0}^{p-1}|z|^{2n}\frac{(1-p)_{n}}{n!}(I_{n}+\overline{J_{n}}),
  \end{align}
where $I_{n}$ and $J_{n}$ satisfy the    recurrence formulas
  \begin{align}\label{2.2}&I_{n}=I_{n-1}-p\frac{\int_{0}^{z}z^{n-1}I_{n-1}dz}{z^{n}},\\
  \label{2.3}&J_{n}=J_{n-1}-p\frac{\int_{0}^{z}z^{n-1}J_{n-1}dz}{z^{n}}\qquad n=1,2,...,p-1,
  \end{align}$I_{0}=h(z)$, and $J_{0}=g(z)$.
\end{theorem}

\begin{proof}
Let $H(z)=\sum_{k=0}^{\infty}c_{k}F(-\frac{\alpha}{2}, k-\frac{\alpha}{2}; k+1; |z|^{2})z^{k}$ and $G(z)=\sum_{k=1}^{\infty}\overline{c_{-k}}F(-\frac{\alpha}{2}, k-\frac{\alpha}{2}; k+1; |z|^{2})z^{k}$. Then by the assumption and (\ref{1.7}), we have
\begin{align}\label{2.4}
 u(z)=H(z)+\overline{G(z)}.
 \end{align}
When $\frac{\alpha}{2}=p-1$, rewrite $H(z)$ as
\begin{align}\label{2.5}
&H(z)=\sum_{k=0}^{\infty}c_{k}F(1-p,k+1-p; k+1; |z|^{2})z^{k}\nonumber\\
  &\qquad\,=\sum_{k=0}^{\infty}c_{k}z^{k}\left(\sum_{n=0}^{\infty}\frac{(1-p)_{n}(k+1-p)_{n}}{(k+1)_{n}}\frac{|z|^{2n}}{n!}\right)\nonumber\\
  &\qquad\,=\sum_{k=0}^{\infty}c_{k}z^{k}\left(\sum_{n=0}^{p-1}\frac{(1-p)_{n}(k+1-p)_{n}}{(k+1)_{n}}\frac{|z|^{2n}}{n!}\right)\nonumber\\
  &\qquad\,=\sum_{n=0}^{p-1}|z|^{2n}\frac{(1-p)_{n}}{n!}I_{n},
  \end{align}
where $$I_{n}=\sum_{k=0}^{\infty}\frac{(k+1-p)_{n}}{(k+1)_{n}}c_{k}z^{k}.$$
Because
\begin{align*}&\frac{(k+1-p)_{n}}{(k+1)_{n}}=\frac{(k+1-p)_{n-1}(k+n-p)}{(k+1)_{n-1}(k+n)}\\
&\qquad\qquad\quad\,\,\,=\frac{(k+1-p)_{n-1}}{(k+1)_{n-1}}-\frac{p}{k+n}\frac{(k+1-p)_{n-1}}{(k+1)_{n-1}},
  \end{align*}
  we can get
  \begin{align*}&I_{n}=\sum_{k=0}^{\infty}\frac{(k+1-p)_{n-1}}{(k+1)_{n-1}}c_{k}z^{k}-\sum_{k=0}^{\infty}\frac{p}{k+n}\frac{(k+1-p)_{n-1}}{(k+1)_{n-1}}c_{k}z^{k}\\
  &\quad=I_{n-1}-p\frac{\int_{0}^{z}z^{n-1}I_{n-1}dz}{z^{n}}.
  \end{align*} This is (\ref{2.2}).

  Similarly, we can get
  \begin{align}\label{2.6}
  G(z)=\sum_{n=0}^{p-1}|z|^{2n}\frac{(1-p)_{n}}{n!}J_{n},
  \end{align}where $J_{n}$ is defined as in
  (\ref{2.3}). Therefore, equation (\ref{2.1}) follows from equations (\ref{2.4})-(\ref{2.6}).
\end{proof}

\begin{example}
From the recurrence formula (\ref{2.1}), we have the following:

(i)\, When $\alpha=0$, i.e. $p=1$,
\begin{align*}u(z)=v(z);
  \end{align*}

(ii)\, When $\alpha=2$, i.e. $p=2$,
\begin{align}\label{2.7}
&u(z)=h+\bar{g}-|z|^{2}\left(h-2\frac{\int_{0}^{z}h(z)dz}{z}+\overline{g-2\frac{\int_{0}^{z}g(z)dz}{z}}\right)\nonumber\\
&\qquad=\sum_{k=0}^{\infty}c_{k}z^{k}+\sum_{k=1}^{\infty}c_{-k}\bar{z}^{k}-|z|^{2}\left(\sum_{k=0}^{\infty}c_{k}\frac{k-1}{k+1}z^{k}+\sum_{k=1}^{\infty}c_{-k}\frac{k-1}{k+1}\bar{z}^{k}\right);
  \end{align}

(iii) \,When $\alpha=4$, i.e. $p=3$,
\begin{align*}&u(z)=h+\bar{g}-2|z|^{2}\left(h-3\frac{\int_{0}^{z}h(z)dz}{z}+\overline{g-3\frac{\int_{0}^{z}g(z)dz}{z}}\right)\\
&\qquad\quad+|z|^{4}\left(h-3\frac{\int_{0}^{z}h(z)dz}{z}-3\frac{\int_{0}^{z}zh(z)dz}{z^{2}}+9\frac{\int_{0}^{z}\int_{0}^{z}h(z)dzdz}{z^{2}}\right.\\
&\qquad\quad\left.+\overline{g-3\frac{\int_{0}^{z}g(z)dz}{z}-3\frac{\int_{0}^{z}zg(z)dz}{z^{2}}+9\frac{\int_{0}^{z}\int_{0}^{z}g(z)dzdz}{z^{2}}}\right)\\
  &\qquad=\sum_{k=0}^{\infty}c_{k}z^{k}+\sum_{k=1}^{\infty}c_{-k}\bar{z}^{k}-2|z|^{2}\left(\sum_{k=0}^{\infty}c_{k}\frac{k-2}{k+1}z^{k}+\sum_{k=1}^{\infty}c_{-k}\frac{k-2}{k+1}\bar{z}^{k}\right)\\
   &\quad\qquad+|z|^{4}\left(\sum_{k=0}^{\infty}c_{k}\frac{(k-1)(k-2)}{(k+1)(k+2)}z^{k}
   +\sum_{k=1}^{\infty}c_{-k}\frac{(k-1)(k-2)}{(k+1)(k+2)}\bar{z}^{k}\right).
  \end{align*}

\end{example}

\section{Lipschitz continuity}
\setcounter{equation}{0}
\hspace{2mm}

\begin{theorem}\label{thm 3.1}
Let $u(z)$ be the corresponding real kernel $\alpha-$harmonic mapping  of $v(z)=h+\bar{g}$ on the unit disk $\mathbb{D}$. If $v(z)$ is Lipschitz continuous on the unit disk $\mathbb{D}$ and $\frac{\alpha}{2}=p-1$ is a nonnegative integer, then $u$ is Lipschitz continuous on the unit disk $\mathbb{D}$ as well.
\end{theorem}

\begin{proof}
By  the assumption and (\ref{2.1}), it is sufficient to prove that $I_{n} $ and $J_{n}$ are  Lipschitz continuous on the unit disk $\mathbb{D}$  for $n=0, 1,2, ...,p-1$. In the following, we just prove the  Lipschitz continuity of $I_{n}$. The case of $J_{n}$ is similar.

Observe that $I_{0}=h(z)$ is holomorphic on $\mathbb{D}$. Then by the recurrence formula (\ref{2.2}), it is easy to see  that all $I_{n}$ are holomorphic on $\mathbb{D}$. It follows that all $I'_{n}$ are holomorphic on  $\mathbb{D}$ too,
where
 \begin{align}\label{3.1}
 I'_{n}=I'_{n-1}-p\frac{z^{n}I_{n-1}-n\int_{0}^{z}z^{n-1}I_{n-1}dz}{z^{n+1}}, \qquad n=1,2,...,p-1.
  \end{align}
Taking account of the maximum modulus principle of holomorphic functions, from equations (\ref{2.2}) and (\ref{3.1}), we get

\begin{align*}&\sup_{z\in\mathbb{D}}|I_{n}|\leq \sup_{z\in\mathbb{D}}|I_{n-1}|+p \sup_{z\in\mathbb{D}}|I_{n-1}|=(p+1) \sup_{z\in\mathbb{D}}|I_{n-1}|
\end{align*}and
\begin{align*}
&\sup_{z\in\mathbb{D}}|I'_{n}|\leq \sup_{z\in\mathbb{D}}|I'_{n-1}|+p \sup_{z\in\mathbb{D}}|I_{n-1}|+np\sup_{z\in\mathbb{D}}|I_{n-1}|
=\sup_{z\in\mathbb{D}}|I'_{n-1}|+(n+1)p\sup_{z\in\mathbb{D}}|I_{n-1}|,
  \end{align*}respectively.
  It follows  that
  \begin{align*}&\sup_{z\in\mathbb{D}}|I_{n}|\leq (p+1)^{n} \sup_{z\in\mathbb{D}}|I_{0}|\end{align*}and
\begin{align}&\sup_{z\in\mathbb{D}}|I'_{n}|\leq \sup_{z\in\mathbb{D}}|I'_{n-1}|+(n+1)p\sup_{z\in\mathbb{D}}|I_{n-1}|\nonumber\\
&\quad\qquad\leq \sup_{z\in\mathbb{D}}|I'_{n-2}|+np\sup_{z\in\mathbb{D}}|I_{n-2}|+(n+1)p\sup_{z\in\mathbb{D}}|I_{n-1}|\nonumber\\
&\quad\qquad\leq ...\nonumber\\
&\quad\qquad\leq \sup_{z\in\mathbb{D}}|I'_{0}|+p\sum_{i=1}^{n}(i+1)\sup_{z\in\mathbb{D}}|I_{i-1}|\nonumber\\
&\label{3.2}\quad\qquad\leq \sup_{z\in\mathbb{D}}|I'_{0}|+p\sum_{i=1}^{n}(i+1)(p+1)^{i-1}\sup_{z\in\mathbb{D}}|I_{0}|.
  \end{align}
 Because $v=h+\bar{g}$ is Lipschitz, there exists a constant $M$ such that
  \begin{align}\label{3.3}
|h'|=|I'_{0}|\leq M
\end{align}for $z\in\mathbb{D}$. It follows that
  \begin{align}\label{3.4}
\sup_{z\in\mathbb{D}}|I_{0}|=\sup_{z\in\mathbb{D}}|h|\leq M.
\end{align}
Therefore, by  inequalities (\ref{3.2})-(\ref{3.4}),  we get that there exists a constant $C=C(M, p, n)$,
such that    \begin{align*}
&\sup_{z\in\mathbb{D}}|I'_{n}|\leq (1+p\sum_{i=1}^{n}(i+1)(p+1)^{i-1})M=:C
\end{align*}for $n=1,2,...,p-1$. It means that $I_{n}$ is Lipschitz continuous on $\mathbb{D}$.

\end{proof}

\section{Univalency of a subclass of  real kernel $\alpha-$harmonic mappings}
\setcounter{equation}{0}
\hspace{2mm}

In the rest of this paper, we use the following notations.
Let $\alpha>-1$, $z=re^{i\theta}$, and
\begin{align*}&t=|z|^{2}=r^{2},\\
&F=F_{k}=F(-\frac{\alpha}{2},k-\frac{\alpha}{2};k+1;t),\quad k=1,2,....,\\
&F_{t}=F_{k,t}=F_{k,t}(-\frac{\alpha}{2},k-\frac{\alpha}{2};k+1;t)=\frac{dF_{k}}{dt}=\frac{dF}{dt}.
 \end{align*}
Furthermore, let \begin{align*}
F_{k}(1)=\lim_{t\rightarrow 1^{-}}F(-\frac{\alpha}{2},k-\frac{\alpha}{2};k+1;t).
\end{align*}
Then by (\ref{1.6}), we have \begin{align}\label{4.1}
F_{k}(1)=\frac{\Gamma(k+1)\Gamma(1+\alpha)}{\Gamma(k+1+\frac{\alpha}{2})\Gamma(1+\frac{\alpha}{2})}.
\end{align}

\begin{lemma}\label{lem4.1}
Let $r_{n}$ and $s_{n}$\, $(n=0,1,2,...)$ be real numbers, and let the power series $$R(x)=\sum_{n=0}^{\infty}r_{n}x^{n} \quad \mbox{and}\quad S(x)=\sum_{n=0}^{\infty}s_{n}x^{n}$$
be convergent for $|x|<r$, $(r>0)$ with $s_{n}>0$ for all $n$. If the non-constant sequence $\{r_{n}/s_{n}\}$ is  increasing (decreasing) for all $n$, then the function $x\mapsto R(x)/S(x)$ is strictly increasing (resp. decreasing) on $(0,r)$.
\end{lemma}

Lemma \ref{lem4.1} is basically due to \cite{BK} (see also \cite{MR3327005}) and in this form with a general setting was stated in \cite{PV}
along with many applications which were later adopted by a number of researchers.

\begin{lemma}\cite{MR4365391} \label{lem4.2}
 Let $\frac{\alpha}{2}\in(0,1]$. Then  it holds that
 \begin{align*}
&(1)\quad \frac{F_{k}}{F_{1}}\leq1 \,\, \mbox{for} \,\, k=1,2,3,... \,\,\mbox{and} \,\,t\in[0,1);\\
&(2)\quad \frac{|F_{k,t}|}{F_{1}}<\frac{(k-\frac{\alpha}{2})\Gamma(k+1)\Gamma(2+\frac{\alpha}{2})}{2\Gamma(k+1+\frac{\alpha}{2})}\,\, \mbox{for}\,\, k=1, 2,3,...\mbox{and}\,\,t\in(0,1).
\end{align*}
\end{lemma}

 \begin{theorem}\label{thm4.3}
 If $\alpha\in(0,2]$, $c_{-k}\in (-N, N)$, where
 \begin{align}\label{4.5}
 N=\frac{\alpha}{2\left(\frac{(k-\frac{\alpha}{2})\Gamma(k+1)\Gamma(2+\frac{\alpha}{2})}
  {\Gamma(k+1+\frac{\alpha}{2})}+k\right)},
\end{align}then the real kernel $\alpha-$harmonic mapping

 \begin{align}\label{4.6}
 u(z)=F_{1}z+c_{-k}F_{k}\overline{z}^{k}, \quad k=1,2,3...,
\end{align} is sense-preserving univalent in $\mathbb{D}$.

 \end{theorem}

 \begin{proof}

 We divide the proof into two steps.

 {\bf First step }: Formula (\ref{4.6}) implies that
 \begin{align*} u_{z}=F_{1}+F_{1,t}t+c_{-k}F_{k,t}\bar{z}^{k+1}, \quad u_{\bar{z}}=F_{1,t}z^{2}+c_{-k}(F_{k,t}z\bar{z}^{k}+kF_{k}\bar{z}^{k-1}).
 \end{align*}It follows  that
  \begin{align*}
 &\quad |u_{z}|-|u_{\bar{z}}|\\
 &\geq F_{1}-|F_{1,t}t|-|c_{-k}F_{k,t}\bar{z}^{k+1}|-|F_{1,t}z^{2}|-|c_{-k}F_{k,t}z\bar{z}^{k}|-k|c_{-k}F_{k}\bar{z}^{k-1}|\\
 &>F_{1}-|F_{1,t}|-|c_{-k}||F_{k,t}|-|F_{1,t}|-|c_{-k}||F_{k,t}|-k|c_{-k}||F_{k}|\\
 &=F_{1}\left[1-\frac{2|F_{1,t}|}{F_{1}}-|c_{-k}|\left(\frac{2|F_{k,t}|}{F_{1}}+k\frac{|F_{k}|}{F_{1}}\right)\right]\\
  &>F_{1}\left[1-(1-\frac{\alpha}{2})-|c_{-k}|\left(\frac{(k-\frac{\alpha}{2})\Gamma(k+1)\Gamma(2+\frac{\alpha}{2})}
  {\Gamma(k+1+\frac{\alpha}{2})}+k\right)\right]\\
  &>0
 \end{align*}for $c_{-k}\in (-N, N)$. The third inequality of  the above holds because of Lemma \ref{lem4.2}. Therefore, $u(z)$ is  sense-preserving.

{\bf Second step }: Let  $c_{-k}=|c_{-k}|e^{i\beta}$.  By assumption,  we have  $\beta=0$ or $\pi$. Let $z=re^{i\theta}$ and
$ u(z)=Re^{i\varphi}$. Rewrite $u(z)$ of (\ref{4.6})  as
\begin{align}&u(z)=F_{1}re^{i\theta}+|c_{-k}|F_{k}r^{k}e^{i(\beta-k\theta)}\nonumber\\
&\label{4.7}\qquad=F_{1}r\cos\theta+|c_{-k}|F_{k}r^{k}\cos (\beta-k\theta)+i(F_{1}r\sin\theta+|c_{-k}|F_{k}r^{k}\sin (\beta-k\theta)).
\end{align}
Then

\begin{align}\label{4.8}
\tan \varphi=\frac{F_{1}r\sin\theta+|c_{-k}|F_{k}r^{k}\sin (\beta-k\theta)}{F_{1}r\cos\theta+|c_{-k}|F_{k}r^{k}\cos (\beta-k\theta)},
\end{align}where $\varphi$ is the argument of $u(z)$. It follows that

\begin{align}
\frac{d}{d\theta}(\tan \varphi)&=\frac{d}{d\theta}\left(\frac{F_{1}r\sin\theta+|c_{-k}|F_{k}r^{k}\sin (\beta-k\theta)}{F_{1}r\cos\theta+|c_{-k}|F_{k}r^{k}\cos (\beta-k\theta)}\right)\nonumber\\
&=\frac{F_{1}^{2}-|c_{-k}|^{2}F_{k}^{2}r^{2(k-1)}k-(k-1)|c_{-k}|F_{1}F_{k}r^{k-1}\cos(\beta-(k+1)\theta)}{[F_{1}\cos\theta+|c_{-k}|F_{k}r^{k-1}\cos (\beta-k\theta)]^{2}}\nonumber\\
&\geq\frac{F_{1}^{2}-|c_{-k}|^{2}F_{k}^{2}r^{2(k-1)}k-(k-1)|c_{-k}|F_{1}F_{k}r^{k-1}}{[F_{1}\cos\theta+|c_{-k}|F_{k}r^{k-1}\cos (\beta-k\theta)]^{2}}\nonumber\\
&=\frac{(F_{1}+|c_{-k}|F_{k}r^{k-1})(F_{1}-|c_{-k}|kF_{k}r^{k-1})}{[F_{1}\cos\theta+|c_{-k}|F_{k}r^{k-1}\cos (\beta-k\theta)]^{2}}\nonumber\\
&\label{4.9}>0
\end{align}
for $|c_{-k}|<\frac{1}{k}$. The last inequality of the above holds because of Lemma \ref{lem4.2} (1). That is to say, $\tan\varphi$ is strictly increasing with respect to $\theta$. So is $\varphi$, too.

In the following we divide into two cases to discuss.

{\bf Case 1} $\beta=0$.  It follows from (\ref{4.8}) that
  \begin{align}\label{4.10}
  \cot\varphi=\frac{\cos\theta+|c_{-k}|\frac{F_{k}}{F_{1}}r^{k-1}\cos k\theta}{\sin\theta-|c_{-k}|\frac{F_{k}}{F_{1}}r^{k-1}\sin k\theta}.
  \end{align} Let's take a close look at the changes in the value of the function  $\cot\varphi$. Firstly, as is well-known, it is easy to verify by
mathematical induction that
\begin{align}\label{4.11}
\left|\frac{\sin k\theta}{\sin\theta}\right|\leq k
\end{align}for $k=1,2,...$ and $\theta\in[0,2\pi)$. If $|c_{-k}|<\frac{1}{k}$, $\alpha\in(0,2]$ and  $\sin\theta\neq 0$, then Lemma \ref{lem4.2} (1) and inequality (\ref{4.11}) imply that $|\sin\theta|>|c_{-k}|\frac{F_{k}}{F_{1}}r^{k-1}|\sin k\theta|$. So, $\sin\theta\neq 0$ implies  $\sin\theta-|c_{-k}|\frac{F_{k}}{F_{1}}r^{k-1}\sin k\theta\neq 0$. In another words, the zero of the denominator of  the right side of equation (\ref{4.10}) comes only from the zero of $\sin\theta$.
Secondly,  $\sin\theta$ only have two zeros in the intervals $[0, 2\pi)$. That is $\theta=0$ and $\pi$. By (\ref{4.10}), we have that
if $\theta=0^{+}$, then  $\cot\varphi=+\infty$; if $\theta=\pi^{-}$, then $\cot\varphi=-\infty$; if $\theta=\pi^{+}$, then $\cot\varphi=+\infty$;  if $\theta=2\pi^{-}$, then $\cot\varphi=-\infty$.
Therefore, considering the continuity and monotonicity of $\cot\varphi$, we can get that the $u(re^{i\theta})$ maps every
circle $|z|=r<1$ in a one-to-one manner onto a  closed Jordan curve.

{\bf Case 2} $\beta=\pi$. Considering (\ref{4.8}), we have
  \begin{align*}\cot\varphi=\frac{\cos\theta-|c_{-k}|\frac{F_{k}}{F_{1}}r^{k-1}\cos k\theta}{\sin\theta+|c_{-k}|\frac{F_{k}}{F_{1}}r^{k-1}\sin k\theta}.
  \end{align*} Follow the discussion of  Case 1. We omit the further details.

It is easy to see that $N<\frac{1}{k}$, where $N$ defined by (\ref{4.5}). Therefore, considering the above two steps of the proof, by degree principle\cite{MR1789408}, we can get that  $u(z)$ is univalent in $\mathbb{D}$.

\end{proof}

The  following  is the well known Rad\'{o}-Kneser-Choquet Theorem, which can be seen in the page 29 of  \cite{MR2048384}.

\begin{theorem}  \quad If $\Omega\in\mathbb{C}$ is a bounded convex domain whose boundary is a Jordan curve $\gamma$ and $f$ is a homeomorphism of the unit circle $\mathbb{T}$ onto $\gamma$, then its harmonic extension
\begin{align*}
u(z)=\frac{1}{2\pi}\int_{0}^{2\pi}\frac{1-|z|^{2}}{|e^{it}-z|^{2}}f(e^{it})dt
\end{align*}
is univalent in $\mathbb{D}$ and defines a harmonic mapping of $\mathbb{D}$ onto $\Omega$.
\end{theorem}

Next, we want to explore the Rad\'{o}-Kneser-Choquet  type theorem  for  real kernel $\alpha-$harmonic mappings. We need the following Proposition at first.

\begin{proposition}\label{prop4.5}
 Suppose $\alpha>-1$ and $c_{-k}\in \mathbb{R}$. Let\begin{align}\label{4.12}
f(e^{i\theta})=F_{1}(1)e^{i\theta}+c_{-k}F_{k}(1)e^{-ik\theta},\quad k=1,2,3,....
\end{align}Then $f$ maps the unit circle $\mathbb{T}$ onto a convex Jordan curve if and only if $c_{-k}\in (-M, M)$, where
 \begin{align}\label{4.13}
 M= \frac{\Gamma(k+1+\frac{\alpha}{2})}{k^{2}\Gamma(k+1)\Gamma(2+\frac{\alpha}{2})}.\end{align}
\end{proposition}

\begin{proof} Direct computation leads to
\begin{align*} \frac{d}{d\theta}(f(e^{i\theta}))=-F_{1}(1)\sin\theta-kc_{-k}F_{k}(1)\sin k\theta
+i(F_{1}(1)\cos\theta-kc_{-k}F_{k}(1)\cos k\theta).
\end{align*}Let $\psi=\psi(\theta)=\arg\{\frac{d}{d\theta}f(e^{i\theta})\}$. Then we have
\begin{align*}
\frac{d}{d\theta}(\tan \psi(\theta))&=\frac{(F_{1}(1))^{2}-k^{3}(c_{-k}F_{k}(1))^{2}+k(k-1)c_{-k}F_{1}(1)F_{k}(1)\cos((k+1)\theta)}{(F_{1}(1)\sin\theta+kc_{-k}F_{k}(1)\sin k\theta)^{2}}\\
&\geq\frac{(F_{1}(1))^{2}-k^{3}(c_{-k}F_{k}(1))^{2}-k(k-1)|c_{-k}|F_{1}(1)F_{k}(1))}{(F_{1}(1)\sin\theta+kc_{-k}F_{k}(1)\sin k\theta)^{2}}\\
&=\frac{(F_{1}(1)+k|c_{-k}|F_{k}(1))(F_{1}(1)-k^{2}|c_{-k}|F_{k}(1))}{(F_{1}(1)\sin\theta+kc_{-k}F_{k}(1)\sin k\theta)^{2}}
\end{align*}
Hence, $\frac{d}{d\theta}(\tan \psi(\theta))\geq 0$ if and only if $|c_{-k}|\leq \frac{F_{1}(1)}{k^{2}F_{k}(1)}= \frac{\Gamma(k+1+\frac{\alpha}{2})}{k^{2}\Gamma(k+1)\Gamma(2+\frac{\alpha}{2})}$.
\end{proof}

Now let $f(e^{i\theta})$ be defined as in (\ref{4.12}) with $\alpha\in (0,2]$, $c_{-k}\in (-L, L)$, where $L=\min\{M, N\}$.  Observe that $\lim_{r\rightarrow 1}u(z):=u^{*}(e^{i\theta})=f(e^{i\theta})$, where $u(z)$  are defined by (\ref{4.6}).
Similar to the second step of the proof of Theorem \ref{thm4.3}, we can verify that $f(e^{i\theta})$ maps unit circle $\mathbb{T}$ onto a  closed Jordan curve in a one-to-one manner, too.
Therefore, considering Theorem 3.3 of \cite{MR3233580} and Theorem \ref{thm4.3} of the above, we actually  get a Rad\'{o}-Kneser-Choquet  type theorem as follows:

 \begin{proposition}\label{pro4.6}
 Let $u^{*}(e^{i\theta})=f(e^{i\theta})$  be defined by (\ref{4.12}) with   $k=1,2,3,...$, $\alpha\in (0,2]$, $c_{-k}\in (-L, L)$, where $L=\min\{M, N\}$, $N$ and $M$ are defined by (\ref{4.5}) and (\ref{4.13}), respectively .  Then  $u^{*}(e^{i\theta})$ is a homeomorphism of the unit circle $\mathbb{T}$ onto a convex Jordan curve $\gamma$ which is a boundary of a bounded convex domain $\Omega\subset\mathbb{C}$. Furthermore,
 $u(z)$ defined by (1.3)  defines a  univalent real kernel $\alpha$-harmonic mapping of  $\mathbb{D}$ onto $\Omega$.
 \end{proposition}

Let us have a look at some special cases of Theorem \ref{thm4.3} or Proposition \ref{pro4.6}.

\begin{example}
Let $\alpha=2$. Then $M=\frac{k+1}{2k^{2}}$ and  $N=\frac{k+1}{k^{2}+3k-2}$. Formula
(\ref{4.12}) deduces to
 \begin{align*}f(e^{i\theta})=e^{i\theta}+\frac{2}{k+1}c_{-k}e^{-ik\theta}.
 \end{align*}Furthermore, let $u^{*}(e^{i\theta})=f(e^{i\theta})$. Then (\ref{1.3}),  or  (\ref{2.7}),  implies that the corresponding real kernel $\alpha$-harmonic mapping is
 \begin{align}\label{4.14}
 u(z)=F_{1}z+c_{-k}F_{k}\bar{z}^{k}=z+c_{-k}\left(1-\frac{k-1}{k+1}|z|^{2}\right)\bar{z}^{k}.
 \end{align} Actually,  it is biharmonic.

(1)\quad If $k=1$ or $k=2$, then $L=M=N$.  If $c_{-k}\in (-L, L)$, then Proposition \ref{pro4.6} says that the $u(z)$  given by (\ref{4.14}) is   univalent,  and $u(\mathbb{D})=\Omega$ is a convex domain.

(2)\quad  If $k=3,4,5,...$, then a direct computation leads to $N>M$. Taking $c_{-k}\in (M, N)$,  Theorem  \ref{thm4.3} and Proposition \ref{prop4.5} imply that  the above $u(z)$ is  still univalent, but $u(\mathbb{D})=\Omega$ is not a convex domain.
\end{example}

\section{Area $S_{u}$}
\setcounter{equation}{0}
\hspace{2mm}

Let $S_{u}$ denote the area of the Riemann surface of $u$. Then we have the following results.

 \begin{theorem}\label{thm5.1}
Let $u$  be a sense-preserving  real kernel $\alpha$-harmonic mapping that has the series expansion of the form (\ref{1.7}) with $c_{0}=0$,  continuous  on $\overline {\mathbb{D}}$. Let $v$ be the corresponding  sense-preserving harmonic mapping that has the series expansion of the form (\ref{1.9}),   continuous  on $\overline {\mathbb{D}}$.
 If  $|c_{k}|\geq|c_{-k}|$ for $k=1,2,...$, then

(1)\quad \, $S_{u}< S_{v}$ for $\alpha\in(0, 2)$ and
      $S_{u}> S_{v}$ for $\alpha\in (-1,0)$;

 (2) \quad \, $S_{u}(\alpha)$ is strictly decreasing with respect to $\alpha\in(-1,0.8)$, where $S_{u}(\alpha)=S_{u}$.

 \end{theorem}
\begin{proof}

By  (\ref{1.7}),  direct computation leads to
 \begin{align*}
u_{z}
&=\sum_{k=1}^{\infty}c_{k}[F_{t}\bar{z}z^{k} +kFz^{k-1}]+\sum_{k=1}^{\infty}c_{-k}F_{t}\bar{z}^{k+1}\\
&=\sum_{k=1}^{\infty}c_{k}[F_{t}r^{k+1}+kFr^{k-1}]e^{i(k-1)\theta}+\sum_{k=1}^{\infty}c_{-k}F_{t}r^{k+1}e^{-i(k+1)\theta}\\
\end{align*}and
 \begin{align*}
u_{\bar{z}}
&=\sum_{k=1}^{\infty}c_{k}F_{t}z^{k+1}+\sum_{k=1}^{\infty}c_{-k}[F_{t}z\bar{z}^{k}+kF\bar{z}^{k-1}]\\
&=\sum_{k=1}^{\infty}c_{-k}[F_{t}r^{k+1}+kFr^{k-1}]e^{-i(k-1)\theta}+\sum_{k=1}^{\infty}c_{k}F_{t}r^{k+1}e^{i(k+1)\theta}.
\end{align*}
So,
\begin{align}
S_{u}&=\int_{0}^{2\pi}\int_{0}^{1}J_{u}(z)rdrd\theta\nonumber\\
&=2\pi\int_{0}^{1}\sum_{k=1}^{\infty}\left[|c_{k}(F_{t}r^{k+1}+kFr^{k-1})|^{2}+|c_{-k}F_{t}r^{k+1}|^{2}-|c_{-k}(F_{t}r^{k+1}+kFr^{k-1})|^{2}-|c_{k}F_{t}r^{k+1}|^{2}\right]rdr\nonumber\\
&=2\pi\int_{0}^{1}\left[\sum_{k=1}^{\infty}(|c_{k}|^{2}-|c_{-k}|^{2})(k^{2}F^{2}r^{2k-1}+2kFF_{t}r^{2k+1})\right]dr\nonumber\\
&=2\pi\sum_{k=1}^{\infty}\left[(|c_{k}|^{2}-|c_{-k}|^{2})k\int_{0}^{1}(kF^{2}r^{2k-1}+2FF_{t}r^{2k+1})dr\right]\nonumber\\
&=\pi\sum_{k=1}^{\infty}\left[(|c_{k}|^{2}-|c_{-k}|^{2})k\int_{0}^{1}d(F^{2}r^{2k})\right]\nonumber\\
\label{5.1}&=\pi\frac{\Gamma^{2}(1+\alpha)}{\Gamma^{2}(1+\frac{\alpha}{2})}\sum_{k=1}^{\infty}\left[k(|c_{k}|^{2}-|c_{-k}|^{2})\frac{\Gamma^{2}(k+1)}{\Gamma^{2}(k+1+\frac{\alpha}{2})}\right].
\end{align}
The last equality holds because of (\ref{1.6}).

Similarly, we have
\begin{align}\label{5.2}
S_{v}=\pi\sum_{k=1}^{\infty}k(|c_{k}|^{2}-|c_{-k}|^{2}).\end{align}

(1)\quad  Let $\psi(x)=\Gamma'(x)/\Gamma(x)$ be the digamma function. Then it is well known that(cf. [3]) $\psi(x)$ is  strictly increasing on $(0, +\infty)$.

Let \begin{align*}
f(x)=\frac{\Gamma(1+\alpha)\Gamma(x+1)}{\Gamma(1+\frac{\alpha}{2})\Gamma(x+1+\frac{\alpha}{2})}.
\end{align*}
Then  we have \begin{align*}
(\log f(x))' &=\psi(x+1)-\psi(x+1+\frac{\alpha}{2}).
\end{align*}
It follows that $(\log f(x))'<0$  provided $\alpha>0$, and  $(\log f(x))'>0$  provided  $\alpha<0$.
Observe that \begin{align*} f(\alpha/2)=1.
\end{align*}Therefore, for $k=1,2,...$, we have $f(k)<1$ if $\alpha\in (0,2)$  as well as $f(k)>1$ if $\alpha\in (-1,0)$. Taking account of (\ref{5.1}) and (\ref{5.2}), we can get Theorem \ref{thm5.1} (1).

(2)\quad As to digamma function $\psi(x)$, we have (cf. [3])
\begin{align}\label{5.3}&\psi(1+x)=\frac{1}{x}+\psi(x),\\
\label{5.4}&\psi(x)=-\gamma+\sum^{\infty}_{n=0}\left(\frac{1}{n+1}-\frac{1}{n+x}\right),
\end{align}and
\begin{align}\label{5.5}
&\psi'(x)=\sum^{\infty}_{n=0}\frac{1}{(x+n)^{2}}
\end{align}
for any $x\in(0, +\infty)$, where $\gamma$ is the Euler-Mascheroni constant.

Let $$h(\alpha)=\psi(1+\alpha)-\psi(1+\frac{\alpha}{2})-\frac{1}{2+\alpha}.$$Then (\ref{5.5}) implies that
\begin{align*}
&h'(\alpha)=\psi'(1+\alpha)-\frac{1}{2}\psi'(1+\frac{\alpha}{2})+\frac{1}{(2+\alpha)^{2}}\\
&\qquad\,=\sum^{\infty}_{n=0}\left(\frac{1}{(1+\alpha+n)^{2}}-\frac{1}{2(1+\frac{\alpha}{2}+n)^{2}}\right)+\frac{1}{(2+\alpha)^{2}}\\
&\qquad\,=\sum^{\infty}_{n=0}\frac{(n+1)^{2}-\frac{\alpha^{2}}{2}}{2(1+\frac{\alpha}{2}+n)^{2}(1+\alpha+n)^{2}}+\frac{1}{(2+\alpha)^{2}}\\
&\qquad\,>0
\end{align*}
for $\alpha\in(-1,1)$. Furthermore, using (\ref{5.4}), direct numerical computation shows
\begin{align*}h(0.8)=-0.0108<0.
\end{align*}Thus, $h(\alpha)<0$ for $\alpha\in (-1,0.8)$.
Let
$$g(\alpha)=\frac{\Gamma(1+\alpha)}{\Gamma(1+\frac{\alpha}{2})\Gamma(k+1+\frac{\alpha}{2})}, \quad k=1,2,....$$ Then it follows that
\begin{align*}
&\frac{d \log g(\alpha)}{d \alpha}=\psi(1+\alpha)-\frac{1}{2}\psi(1+\frac{\alpha}{2})-\frac{1}{2}\psi(k+1+\frac{\alpha}{2})\\
&\qquad\qquad\,<\psi(1+\alpha)-\frac{1}{2}\psi(1+\frac{\alpha}{2})-\frac{1}{2}\psi(2+\frac{\alpha}{2})\\
&\qquad\qquad\,=h(\alpha).
\end{align*}
 That is to say $g(\alpha)$ is strictly decreasing on $(-1,0.8)$. Therefore, (\ref{5.1}) implies that $S_{u}(\alpha)$ is strictly decreasing with respect to $\alpha\in(-1,0.8)$.

\end{proof}

{\bf Acknowledgements}  \quad The authors heartily thank  the anonymous reviewers for their careful review and for their effective suggestions.

{\bf Declarations}\\

{\bf Conflict of interests}\quad The authors declare that they have no conflict of interest.\\

{\bf Data availability statement}\quad My manuscript has no associated date.

\medskip



\end{document}